\documentclass{amsart}
\usepackage{amssymb} 
\usepackage{amsmath} 
\usepackage{mathrsfs} 
\usepackage{eufrak}
\usepackage{amscd}
\usepackage{amsbsy}
\usepackage{comment}
\usepackage[matrix,arrow]{xy}
\usepackage{hyperref}
\usepackage{mathtools}
\usepackage{esint}
\usepackage[letterpaper, margin=1.5in]{geometry}
\usepackage{hyperref}
\usepackage{enumerate}

\newcommand{\N}{{\mathbb N}}

\newcommand{\R}{{\mathbb R}}

\newcommand{\sA}{{\mathcal A}}

\newcommand{\sM}{{\mathcal M}}

\newcommand{\sR}{{\mathcal R}}

\newcommand{\inv}{^{-1}}

\newcommand{\ep}{\epsilon}

\newcommand{\Rm}{\mathrm{Rm}}
\newcommand{\Rc}{\mathrm{Rc}}
\newcommand{\Vol}{\mathrm{Vol}}
\newcommand{\inj}{\mathrm{inj}}

\newcommand{\RNum}[1]{\uppercase\expandafter{\romannumeral #1\relax}}

\let\epsilon\varepsilon

\newcommand{\be}{\begin{equation}}
\newcommand{\ee}{\end{equation}}
\newcommand{\ba}{\begin{align*}}

\newtheorem{thm}{Theorem}[section]
\newtheorem{lem}[thm]{Lemma}
\newtheorem{claim}[thm]{Claim}
\newtheorem{prop}[thm]{Proposition}

\theoremstyle{definition}
\newtheorem{remark}[thm]{Remark}

\def\mod#1{{\ifmmode\text{\rm\ (mod~$#1$)}
\else\discretionary{}{}{\hbox{ }}\rm(mod~$#1$)\fi}}
\begin{document}

\setlength\parindent{0pt}

        \author[A. Martens]{Adam Martens}
        \address{Department of Mathematics, The University of British Columbia,
1984 Mathematics Road, Vancouver, B.C.,  Canada V6T 1Z2}
\email{martens@math.ubc.ca}

        \bibliographystyle{amsplain}
\title{Removing scalar curvature assumption for Ricci flow smoothing}

\maketitle

\pagestyle{plain}

\vspace{-3em}

\begin{abstract}
In recent work of Chan-Huang-Lee, it is shown that if a manifold enjoys uniform bounds on (a) the negative part of the scalar curvature, (b) the local entropy, and (c) volume ratios up to a fixed scale, then there exists a Ricci flow for some definite time with estimates on the solution assuming that the local curvature concentration is small enough initially (depending only on these a priori bounds). In this work, we show that the bound on scalar curvature assumption (a) is redundant. We also give some applications of this quantitative short-time existence, including a Ricci flow smoothing result for measure space limits, a Gromov-Hausdorff compactness result, and a topological and geometric rigidity result in the case that the a priori local bounds are strengthened to be global. 
\end{abstract}
\vspace{5pt}
Mathematics Subject Classification code(s): \textbf{53E20} \emph{Ricci flows}.

\section{Introduction}

In \cite{CHL}, Chan-Huang-Lee prove a quantitative version of short-time Ricci flow existence, with scale-invariant estimates on the solution, starting from a metric $g$ that has effective uniform bounds on the negative part of the scalar curvature, volume ratios up to a fixed scale, and local $\bar{\nu}$-entropy, assuming that the local scale-invariant integral curvature (so-called ``curvature concentration") is uniformly dependently small. Although it is not necessary to assume that $g$ has uniformly bounded sectional curvature, the authors of \cite{CHL} do impose that it should enjoy an \emph{ineffective} lower Ricci bound. To summarize precisely, the initial metric $g$ should satisfy
\be
\begin{cases}\label{CHLassumptions}
\sR_g\geq -\lambda,\\
\Vol_g B_g(x,r)\leq \mathfrak{v}_0 r^n \;\text{ for all } x\in M, r\in (0,1],\\
\bar\nu(B_g(x,1), g, \tau)\geq -A  \;\text{ for all } x\in M,  \\
\left(\int_{B_{g}(x,1)} |\Rm|_g^{\frac{n}{2}}\,dV_g\right)^{\frac{2}{n}}\leq \sigma \;\text{ for all } x\in M,\text{ and }\\
\inf_{x\in M}\Rc_g(x)>-\infty,
\end{cases}
\ee
where $\lambda, \mathfrak{v}_0, \tau, A$ can be prescribed arbitrarily, but $\sigma$ depends on these choices. We direct the reader to \cite{WangA} for the definition and properties of the localized $\nu$- and $\bar\nu$-functionals. The ineffective lower Ricci bound in \eqref{CHLassumptions} is expected not to be necessary as it is only ever used to construct a smooth distance-like function on $M$, which in turn facilitates the approximation of the possibly unbounded curvature metric $g$ by a sequence of bounded curvature metrics satisfying the local assumptions \eqref{CHLassumptions} (with possibly different constants). As commented on in \cite{CHL}, this assumption may be weakened to a condition on how quickly $\Rc_{g}^-$ grows at spatial infinity, though we do not pursue this here. Rather, the purpose of this work is to show that the scalar curvature bound assumption in \eqref{CHLassumptions} is redundant in the sense that the required smallness of $\sigma$ and the estimates of the solution need not depend on $\lambda$. We prove the following. \\

\vspace{5pt}

\begin{thm}\label{shortimeexistencemain}
For all $n\geq 4,A, \tau, \mathfrak{v}_0>0$, there are $C_0, \sigma, T>0$ depending only on $n,A,\tau,\mathfrak{v}_0$ such that the following holds. Let $(M^n,g)$ be a complete noncompact manifold (not necessarily bounded curvature) such that for some $\ep\leq \sigma$ and all $x\in M$, the metric $g$ satisfies: 
\begin{enumerate}[(i)]
\item $\Vol_g B_g(x,r)\leq \mathfrak{v}_0 r^n$ for all $r\leq 1$; \label{shorttimeexistencehyp1}
\item $\bar\nu(B_{g}(x, 1), g, \tau)\geq -A$; \label{shorttimeexistencehyp2}
\item $\left(\int_{B_{g}(x,1)} |\Rm|_g^{\frac{n}{2}}\,dV_g\right)^{\frac{2}{n}}\leq \ep$; and \label{shorttimeexistencehyp3}
\item $\inf_{x\in M}\Rc_g(x)>-\infty$.\label{shorttimeexistencehyp4}
\end{enumerate}
Then there exists a complete Ricci flow $g(t)$ on $M\times [0,T]$ with $g(0)=g$ such that for all $(x,t)\in M\times (0,T]$, there holds
\be\label{shorttimeexistenceconclusions}
\begin{cases}
&|\Rm|_{g(t)}(x)\leq C_0 \ep t\inv, \\
&\inj_{g(t)}(x)\geq C_0\inv \sqrt t,\text{ and}\\
&\left(\int_{B_{g(0)}(x,\frac{1}{16})} |\Rm|_{g(t)}^{\frac{n}{2}}\,dV_{g(t)}\right)^{\frac{2}{n}}\leq C_0\ep.
\end{cases}
\ee
\end{thm}

\vspace{5pt}

Before we provide the applications of Theorem \ref{shortimeexistencemain}, we first make some remarks regarding its hypotheses and statement. 

\begin{itemize}
\item We have already mentioned above that the ineffective lower Ricci bound \eqref{shorttimeexistencehyp4} is likely a redundant hypothesis as it is only used to construct a distance-like function on $(M,g)$ satisfying some desired properties. Of particular note here, we suggest that the upper volume bound assumption \eqref{shorttimeexistencehyp1} (here and throughout the paper) should \emph{also} be redundant. Some evidence for this is discussed after Question 1 below. On the other hand, the $C^0$-control in the form of a $\bar\nu$-functional bound \eqref{shorttimeexistencehyp2} and the proportional smallness of the curvature concentration \eqref{shorttimeexistencehyp3} are a somewhat natural pairing from a Ricci flow point of view (see \cite{Martens} and the references within for more discussion on this).
\item The choice of $r=\frac{1}{16}$ in the curvature concentration in \eqref{shorttimeexistenceconclusions} is for convenience, not necessity; any $r<1$ would work with appropriate modification. Also, by applying the Shrinking Balls Lemma \cite{ST}, the domain in the integral could be replaced with the ball of radius $r$ with respect to $g(t)$ after possibly shrinking $T$ further.
\item The constants $C_0, \sigma, T$ in Theorem \ref{shortimeexistencemain} can be explicitly bounded in terms of $n,A,\tau,\mathfrak{v}_0,$ though we do not pursue this here because our methods for bounding these variables would be far from sharp. For example, suppose $\tau=1$ and $\mathfrak{v}_0=2\omega_n$, just for simplicity. Then, since the local curvature concentration ought to be preserved by the flow (up to a dimensional constant perhaps), we can compute that an optimal bound on $\sigma$, from the parabolic point of view at least, should be
\be\label{expectedboundsigma}
\sigma\geq \frac{1}{C(n)e^{\frac{2A}{n}}}.
\ee
This expected bound on $\sigma$ comes from the equivalence between the $\bar\nu$-functional and an $L^2$-Sobolev inequality, and the relationship between this Sobolev inequality and the curvature concentration via H\"older's inequality. The reason why we would be unable to achieve a bound of the form \eqref{expectedboundsigma} with our current machinery is due to the fact that there is not a known sharp relationship between volume ratios and $L^2$-Sobolev inequalities in spaces with possibly negative Ricci curvature. We direct the reader to sections 1 and 2 of \cite{Martens} for more related discussion. 
\end{itemize}

\vspace{5pt}

We also give two applications of our short-time existence theorem. The first of which can be viewed either (a) as a Ricci flow smoothing result for measure spaces $(M,\mu)$ that are measure-space limits of a sequence of smooth manifolds $(M_i,g_i)$ satisfying Theorem \ref{shortimeexistencemain}, or (b) as a Gromov-Hausdorff compactness result for the family of metrics satisfying the hypotheses of Theorem \ref{shortimeexistencemain}. One important thing to point out about Theorem \ref{smoothingcorollary} is that the only $L^\infty$ assumption on the curvature is that each $g_i$ satisfies \emph{some} uniform lower Ricci bound, but that bound need not be the same for each $i\in\N$.

\vspace{5pt}

\begin{thm}\label{smoothingcorollary}
For any $n\geq 4, A,\tau,\mathfrak{v}_0>0$, let $(M_i^n, g_i, p_i)$ be any sequence of complete pointed Riemannian manifolds (not necessarily bounded curvature) each satisfying hypotheses \eqref{shorttimeexistencehyp1} - \eqref{shorttimeexistencehyp4}. Then there exists a smooth pointed manifold $(M_\infty, x_\infty)$ endowed with a complete distance metric $d_0$ (which generates the same topology as the smooth structure of $M_\infty$) and a Radon measure $\mu$ such that, up to a subsequence, 
\be\label{metricmeasurespaceconvergence}
(M_i, d_{g_i},\Vol_{g_i},p_{i}) \longrightarrow(M_\infty,d_0,\mu, x_\infty)
\ee
where the convergence is as pointed metric measure spaces (as described precisely below). Moreover, there exists a complete Ricci flow $g(t)$ on $M_\infty\times (0,T(n,A,\tau,\mathfrak{v}_0)]$ satisfying the estimates \eqref{shorttimeexistenceconclusions} which emerges from $(M_\infty, \mu)$ in the weak sense, i.e., that for any precompact Borel set $E\subset M_\infty$, there holds 
\[
\lim_{t\to 0}\Vol_{g(t)}(E)=\mu(E).
\]
The notion in which the convergence \eqref{metricmeasurespaceconvergence} holds is as follows. There is some subsequence $\{i_k\}_{k\in \N}$ of $\N$, an exhausting sequence $\Omega_k\subset M_\infty$ of open sets, and functions $F_k: \Omega_k\to M_{i_k}$ (each of which are diffeomorphisms onto their images) with $F_k(x_\infty)=p_{i_k}$ which satisfy
\[
\lim_{k\to \infty} \Vol_{F_k^*g_{i_k}}(E)=\mu(E)
\]
for any precompact Borel set $E\subset M_\infty$, and
\[
(M_{i_k}, d_{g_{i_k}},p_{i_k}) \xrightarrow{\text{pointed Gromov-Hausdorff}}(M_\infty,d_0, x_\infty),
\]
where the $F_k$'s may serve as the approximating functions.
\end{thm}

\begin{remark}
Although the convergence of the distance function $d_{g(t)}$ back to the `time 0' data $d_0$ is not guaranteed by Theorem \ref{smoothingcorollary}, it is certainly expected that this is also true. The shortfall here is due to the fact our version of the Expanding Balls Lemma (Lemma \ref{EBL}) does not give a continuous rate of expansion at the initial time. See also Question 2 below.
\end{remark}

The second application we present is for manifolds $(M,g)$ that satisfy the hypotheses of Theorem \ref{shortimeexistencemain} on \emph{all} scales, thus allowing us to apply the theorem to increasingly extreme blow-downs to obtain long-time Ricci flow existence with the scale invariant properties \eqref{shorttimeexistenceconclusions} holding for \emph{all} times. This allows us to conclude topological rigidity and global volume bounds on the initial manifold. 

\vspace{5pt}

\begin{thm}\label{longtimecorollary}
For all $n\geq 4, A, \mathfrak{v}_0 >0$, there exists $\sigma>0$ depending on $n,A,\mathfrak{v}_0$ such that the following holds. For any $\ep\leq \sigma$, we denote $\sM_\ep$ to be the family of complete noncompact (not necessarily bounded curvature) Riemannian manifolds $(M^n,g)$ satisfying:
\begin{enumerate}[(i)]
\item $\Vol_g B_g(x,r)\leq \mathfrak{v}_0 r^n$ for all $x\in M, r>0$; \label{longtime1}
\item $\bar\nu(M,g)\geq -A$;  \label{longtime2}
\item $\left(\int_M |\Rm|_g^{\frac{n}{2}}\,dV_g\right)^{\frac{2}{n}}\leq \ep$; and  \label{longtime3}
\item $\inf_{x\in M}\Rc_g(x)>-\infty$.  \label{longtime4}
\end{enumerate}
Then each of the following rigidity properties holds:
\begin{enumerate}[(a)]
\item If $(M^n,g)\in \sM_\sigma$, then $M^n$ is homeomorphic to $\R^n$, and if $n\neq 4$ then the homeomorphism can be taken to be a diffeomorphism.\label{longtimecorconc1}
\item For any $(M,g)\in \sM_\sigma$, the volume growth of $(M,g)$ at infinity is at least that of Euclidean space, uniformly with respect to the base point. More precisely,\label{longtimecorconc2}
\[
\liminf_{r\to \infty, x\in M} \frac{\Vol_g B_g(x,r)}{r^n}\geq \omega_n.
\]
\item By restricting $\ep$ smaller, we may assume that \textbf{all} volume ratios are uniformly close (from below) to that of Euclidean space, uniformly with respect to the manifold, base point, and scale. More precisely,\label{longtimecorconc3}
\[
\liminf_{\ep\to 0} \left\{\frac{\Vol_g B_g(x,r)}{r^n}\; \text{ where }  (M,g)\in \sM_\ep, x\in M, r>0\right\}\geq \omega_n .
\]
\end{enumerate}
\end{thm}

\vspace{5pt}

In addition to the obvious question of whether the ineffective lower Ricci bound assumption \eqref{longtime4} can be removed from Theorem \ref{longtimecorollary}, we also present two additional questions. 

\vspace{5pt}

\textbf{Question 1:} Is the upper bound on volume ratios assumption \eqref{longtime1} also redundant in the statement of Theorem \ref{longtimecorollary}?

\vspace{5pt}

For example, if $g$ is assumed to have bounded curvature, then conclusions \eqref{longtimecorconc2} and \eqref{longtimecorconc3} are still true even with assumption \eqref{longtime1} removed, which follows from their methods of proof and the main result of \cite{CM}. Moreover, the main result of \cite{CM} includes a topological rigidity statement stronger than conclusion \eqref{longtimecorconc1}, assuming an additional mild hypothesis on the curvature but without the volume bound assumption. The primary issue that we face without the bounded curvature assumption is that our method to show that the curvature concentration remains bounded along the flow uses the Expanding Balls Lemma \ref{EBL} (EBL), which seems to require an a priori upper bound on volume ratios. 

\vspace{5pt}

Our other question is also related to an issue with EBL. Specifically, if we had a version of this lemma that gave a \emph{continuous} rate of expansion akin to that of the Shrinking Balls Lemma (instead of simply the conclusion that $B_{g(0)}(x,1)\subset B_{g(t)}(x,R)$), then we would be able to prove that the analogous almost-Euclidean upper bound conclusions of \eqref{longtimecorconc2} and \eqref{longtimecorconc3} would also be true. In other words, we would have an affirmative answer to the following. 

\vspace{5pt}

\textbf{Question 2:} With the same notation of $\sM_\ep$ from Theorem \ref{longtimecorollary}, is it true that 
\[
\limsup_{r\to \infty, x\in M} \frac{\Vol_g B_g(x,r)}{r^n}\leq \omega_n,
\]
for any $(M,g)\in \sM_\sigma$ and that 
\[
\limsup_{\ep\to 0} \left\{\frac{\Vol_g B_g(x,r)}{r^n}\; \text{ where }  (M,g)\in \sM_\ep, x\in M, r>0\right\}\leq \omega_n ?
\]

\vspace{5pt}

The organization of this paper is as follows. In section \ref{distancesection}, we prove a modified version of the Expanding Balls Lemma, which substitutes the bound on the scalar curvature for a bound on the curvature concentration along the flow. In section \ref{proofssection}, we prove Theorems \ref{shortimeexistencemain}, \ref{smoothingcorollary}, \ref{longtimecorollary}, each in their own subsection. Some of these proofs are similar to the methods in \cite{CCL}, \cite{CHL}, \cite{CM}, \cite{Martens}, in which case we will only describe the necessary changes in this setting. 

\section{Modified distance distortion}\label{distancesection}

First, we recall the well-known Shrinking Balls Lemma due to Simon-Topping, which dictates how quickly distances can contract under Ricci flows with scale-invariant curvature bounds: 
\[
\sup_{y\in B_{g(0)}(x, r)}|\Rm|_{g(t)}(y)\leq c_0/t \;\Longrightarrow \; B_{g(t)}(x, r-\beta \sqrt{c_0 t})\subset B_{g(0)}(x, r) \tag{SBL}
\]
Here $\beta=\beta(n)$ is a dimensional constant, which is now fixed notation in the remainder of the paper. We direct the reader to {\cite[Corollary 3.3]{ST}} for a precise statement. \\

Obtaining control on how quickly distances can expand under the flow usually requires control on the negative part of the Ricci or scalar curvature along the flow. In our setting, however, we are able to substitute this $L^\infty$ curvature assumption for control on the local curvature concentration along the flow. We begin with an elementary lemma that allows us to control the rate at which volumes are perturbed, which will be useful in its own right. 

\vspace{5pt}

\begin{lem}[Volume Perturbation]\label{volumeperturbation}
For any $n \in \mathbb{N}, v, \sigma, r > 0$, and any pointed (not assumed complete) Ricci flow $(M^n, g(t),x_0)$, $t\in [a,b]$ for some $a<b$, let $\Omega\subset\subset M$ be any precompact set satisfying:
\begin{enumerate}[(a)]
    \item $B_{g(t)}(x_0, r)\subset \Omega$ for each $t\in [a,b]$;\label{volumepertcontainment}
    \item $\left(\int_{\Omega}|\sR_{g(t)}|^{\frac{n}{2}}\,dV_{g(t)}\right)^{\frac{2}{n}} \leq \sigma$ for all $t\in [a,b]$;\label{volumepert1}
    \item $\Vol_{g(t)}B_{g(t)}(x_0,r) \geq v r^n$ for all $t\in [a,b]$.\label{volumepert4}
\end{enumerate}
Then whenever $a\leq s\leq t\leq  b$, there holds 
\[
e^{- \sigma v^{-\frac{2}{n}}r^{-2} (t-s)}\Vol_{g(s)}\Omega\leq \Vol_{g(t)}\Omega\leq e^{ \sigma v^{-\frac{2}{n}}r^{-2}(t-s)}\Vol_{g(s)}\Omega.
\]
\end{lem}

\begin{remark}\label{volpertremark}
Our most common application of Lemma \ref{volumeperturbation} will be with $a=s=0$ and $\Omega=B_{g(0)}(x_0, 2r)$. In this setting, if we further assume that the flow enjoys an $\alpha/t$ curvature bound, then we can see (by SBL) that condition \eqref{volumepertcontainment} can be replaced by the assumption that $B_{g(0)}(x_0, 2r)\subset \Omega$ as long as we also restrict $b\leq \frac{r^2}{\beta^2 \alpha}$. 
\end{remark}

\begin{proof}[Proof of Lemma \ref{volumeperturbation}]
By the evolution of the volume form $\partial_t dV_{g(t)}=-\sR_{g(t)}dV_{g(t)}$, the assumed smallness of integral scalar curvature \eqref{volumepert1} and the lower bound of volumes at time $t$ in \eqref{volumepert4}, we obtain
\ba
\partial_t \Vol_{g(t)}(\Omega)&\leq \left(\int_\Omega |\sR_{g(t)}|^{\frac{n}{2}}dV_{g(t)}\right)^{\frac{2}{n}}[\Vol_{g(t)}(\Omega)]^{\frac{n-2}{n}}\\&
\leq \sigma [\Vol_{g(t)}B_{g(t)}(x_0, r)]^{-\frac{2}{n}}\Vol_{g(t)}(\Omega)\\&\leq 
\sigma (vr^n)^{-\frac{2}{n}}\Vol_{g(t)}(\Omega) .
\end{align*}
Therefore by comparison, we obtain $\Vol_{g(t)}\Omega\leq e^{\sigma v^{-\frac{2}{n}}r^{-2}(t-s)} \Vol_{g(s)}\Omega$ for any $0\leq s\leq t\leq T$. The other inequality follows from an identical argument.
\iffalse 
\ba
\partial_t \Vol_{g(t)}(\Omega)&\geq- \left(\int_\Omega |\sR_{g(t)}|^{\frac{n}{2}}dV_{g(t)}\right)^{\frac{2}{n}}[\Vol_{g(t)}(\Omega)]^{\frac{n-2}{n}}\\&\geq 
-\sigma\Vol_{g(t)}(\Omega) [\Vol_{g(t)}(\Omega)]^{-\frac{2}{n}}\\&\geq - \sigma\Vol_{g(t)}(\Omega) [\Vol_{g(t)}B_{g(t)}(x_0, 1)]^{-\frac{2}{n}}\\&\geq 
-\sigma v^{-\frac{2}{n}}\Vol_{g(t)}(\Omega) .
\end{align*}
\fi
\end{proof}

\vspace{5pt}

This volume perturbation estimate allows us to remove the assumption of a lower bound on scalar curvature in the version of EBL from \cite{Martens}. 

\vspace{5pt}

\begin{lem}[Expanding Balls]\label{EBL}
For any $n \in \mathbb{N}$ and any $\alpha, v, \sigma > 0$, there exists $\widehat{T}(n, \alpha, v, \sigma) > 0$ such that the following holds. Let $(M^n, g(t), x_0)$, $t\in [0,T]$ be a pointed Ricci flow (not assumed complete) such that for some $V>0$ there holds: 
\begin{enumerate}[(a)]
	\item $B_{g(t)}(x_0, 2^{n+2} Vv\inv)\subset \subset M$ for all $t\in [0,T]$;
    \item $\left(\int_{B_{g(0)}(x_0,2)}|\sR_{g(t)}|^{\frac{n}{2}}\,dV_{g(t)}\right)^{\frac{2}{n}} \leq \sigma$ for all $t\in [0,T]$;\label{eblcond1}
    \item $\sup_{x\in B_{g(0)}(x_0,2)}|\Rm|_{g(t)}(x) \leq \alpha t^{-1}$ for all $t\in (0,T]$;
    \item $\Vol_{g(0)}B_{g(0)}(x_0, 2)\leq V $; and 
    \item $\Vol_{g(t)}B_{g(t)}(x,r) \geq v  r^n$ whenever $x\in B_{g(0)}(x_0, 1)$, $r\in (0,\frac{1}{2}]$, and $t\in (0,T]$.\label{eblcond4}
\end{enumerate}
Then for all $t \in [0, \min\{T, \widehat{T}\}]$, we have $B_{g(0)}(x_0, 1)  \subset B_{g(t)}(x_0, 2^{n+2} Vv\inv)$.
\end{lem}

\begin{proof}
The proof is nearly identical to the argument in \cite{Martens}. The only changes are as follows. First, the maximal time interval in our setting will be
\[
\widehat T= \min\left\{\frac{v^{\frac{2}{n}}\log 2 }{\sigma } , \frac{1}{4\beta^2 \alpha}\right\}.
\]
Now the only setting in which the lower bound of scalar curvature was used in {\cite[Theorem 3.2]{Martens}} was to obtain the volume control 
\[
\Vol_{g(t)}(B_{g(0)}(x_0, 2))\leq 2 \Vol_{g(0)}(B_{g(0)}(x_0, 2))
\]
for all $t\in [0,\min\{T, \widehat T\}]$. For us, this implication follows immediately from Lemma \ref{volumeperturbation} by our restrictions on the size of $\widehat T$ and Remark \ref{volpertremark}.
\end{proof}

\section{Proof of results}\label{proofssection}

\subsection{Proof of Theorem \ref{shortimeexistencemain}}\hfill

\vspace{5pt}

We claim that it suffices to show that the theorem holds for complete \emph{bounded curvature} metrics, albeit for slightly different scales (cf. Claim \ref{shorttimeclaim}). The reason why this is sufficient is as follows. We know by {\cite[Theorem 2.1]{CHL}} that there exists a smooth proper function $\rho: M\to \R$ which, for some constant $C$ depending on the metric $g$, satisfies the following conditions uniformly for all $x\in M$:
\be\label{unboundedcurvdistancelikefunctionproperties}
\begin{cases}
&C\inv d_{g}(x_0,x)-C \leq \rho(x)\leq d_{g}(x_0, x) ;\\
& |\nabla \rho(x)|_g^2+|\Delta^g \rho(x)|\leq C; \text{ and }\\
&\int_{B_{g}(x,1)} |\nabla^2 \rho|_g^n \,dV_{g} \leq C .
\end{cases}
\ee
We emphasize again here that the ineffective lower Ricci curvature bound \eqref{shorttimeexistencehyp4} is only used to show that \eqref{unboundedcurvdistancelikefunctionproperties} holds. Once we know that \eqref{unboundedcurvdistancelikefunctionproperties} holds, by the proof of {\cite[Theorem 1.1]{CHL}}, we can say that there exists a sequence of exhausting precompact open sets $U_j\subset M$ and complete bounded curvature metrics $g_j$ on $U_j$ that satisfy all the following:
\begin{itemize}
\item For any compact $\Omega\subset M$, $g_j=g$ on $\Omega$ for all $j$ sufficiently large; 
\item $\Vol_{g_j} B_{g_j}(x,r)\leq 2\mathfrak{v}_0 r^n$ for all $r\leq 1/2$;
\item $\bar\nu(B_{g_j}(x, 1/2), g_j, \tau)\geq -\tilde A(n,A,\tau)$; and
\item $\left(\int_{B_{g_j}(x,1)} |\Rm|_{g_j}^{\frac{n}{2}}\,dV_{g_j}\right)^{\frac{2}{n}}\leq 4\ep$.
\end{itemize}
Define the sequence of blown-up metrics $\tilde g_j=16^2g_j$, which (in particular) satisfies each of the following conditions uniformly with respect to $j\in \N$ and $x\in U_j$:
\begin{enumerate}[(I)]
\item $\Vol_{\tilde g_j} B_{\tilde g_j}(x,r)\leq 2\mathfrak{v}_0 r^n$ for all $r\leq 8$;\label{shorttimeboundedcurvvolume}
\item $\bar\nu(B_{\tilde g_j}(x, 4), \tilde g_j, \tau)\geq -\tilde A$; and
\item $\left(\int_{B_{\tilde g_j}(x,4)} |\Rm|_{\tilde g_j}^{\frac{n}{2}}\,dV_{\tilde g_j}\right)^{\frac{2}{n}}\leq 4\ep$.
\end{enumerate}

Now we claim that the $\tilde g_j$ satisfy a slightly modified version of Theorem \ref{shortimeexistencemain}. 

\begin{claim}\label{shorttimeclaim}
There exists $\widetilde T, C_0, \sigma$ depending only on $n,A, \tau, \mathfrak{v}_0$ such that the complete bounded curvature Ricci flow $\tilde g_j(t)$ with $\tilde g_j(0)=\tilde g_j$ exists at least until time $\widetilde T$ and satisfies
\[
\begin{cases}
&|\Rm|_{\tilde g_j(t)}(x)\leq C_0 \ep t\inv,\\
&\inj_{\tilde g_j(t)}(x)\geq C_0\inv \sqrt t, \;\text{ and}  \\
&\left(\int_{B_{\tilde g_j(0)}(x,\frac{5}{4})} |\Rm|_{\tilde g_j(t)}^{\frac{n}{2}}\,dV_{\tilde g_j(t)}\right)^{\frac{2}{n}}\leq C_0\ep
\end{cases}
\]
uniformly for all $(x,t)\in U_j\times [0,\widetilde T]$, assuming that $\ep\leq \sigma$.
\end{claim}

To prove the theorem (assuming the claim is true), we briefly sketch the argument of the proof of {\cite[Theorem 1.1]{CHL}}. Apply the claim to each $\tilde g_j$ and scale back down to obtain a sequence of flows $g_j(t)$ on $U_j\times [0,T]$, where $T=\frac{\widetilde T}{16^2}$, and each $g_j(t)$ satisfies the estimates \eqref{shorttimeexistenceconclusions}. Since $g_j=g$ on any compact set as $j\to \infty$, we can say by {\cite[Corollary 3.2]{BingLongChen}} and Shi's modified interior estimates {\cite[Theorem 14.16]{Chowetal}} that
\[
\limsup_{j\to \infty} \sup_{\Omega\times [0,T]} |\nabla^k \Rm|_{g_j(t)}\leq C(n,k,\Omega, g), 
\]
for any $k\geq 0$ and any compact $\Omega\subset M$. Then we may apply the Arzel\'a-Ascoli Theorem and a diagonal argument on a countable covering of coordinate charts to conclude that some subsequence of the $g_j(t)$ converges in $C^\infty_{loc}(M\times[0,T])$ to a Ricci flow $g(t)$ on $M\times [0,T]$ with $g(0)=g$. Moreover, $g(t)$ satisfies \eqref{shorttimeexistenceconclusions} because each $g_j(t)$ does, and is thus also a complete flow by SBL. Therefore it suffices to prove the claim. 

\vspace{5pt}

\begin{proof}[Proof of claim]
This argument is only a mild modification of {\cite[Theorem 1.2]{CCL}} and {\cite[Theorem 4.2]{Martens}}. We will follow the proof given in \cite{Martens}, using the same notation. For brevity, we will suppress the dependence on $j$ by denoting $g(t)=\tilde g_j(t)$, $t\in [0,T]$ to be the unique bounded curvature Ricci flow emanating from $\tilde g_j$. Let us begin by recalling the outline of the proof from \cite{Martens}. There the author defines some constants $\widetilde T, \Lambda$ depending only on $n,A,\tau,\mathfrak{v}_0$ (there, the values of $\mathfrak{v}_0$ and $\tau$ were fixed as $2\omega_n$ and $2$ respectively), and then defines $T'>0$ to be the maximal time in $[0, \min\{T, \widetilde T\}]$ for which
\be\label{originalg(t)assumption}
\int_{B_{g(t)}(x, 1)}|\Rm|_{g(t)}^{\frac{n}{2}}\,dV_{g(t)}\leq \Lambda \ep^{\frac{n}{2}}
\ee
uniformly for all $(x,t)\in M\times [0,T']$. Then the proof is split into three steps which we now briefly describe.

\begin{enumerate}[(Step 1)]
\item For any $\alpha>0$, let $T_1$ denote the largest time in which $|\Rm|_{g(t)}\leq \frac{\alpha}{t}$, which is necessarily positive because the flow $g(t)$ is of bounded curvature. Then this step shows the flow enjoys a local $L^2$-Sobolev inequality, and thus also a uniform lower bound on volume ratios up to a fixed scale, for all $(x,t)\in M\times [0,\min\{T', T_1\}]$ where the constants are uniform with respect to the choice of $(x,t)$.
\item This step shows that the $T_1$ from Step 1 can actually be taken to only depend on $n,A,\tau,\mathfrak{v}_0$ as long as $\Lambda$ is sufficiently large (depending only on $n,A,\tau, \mathfrak{v}_0$). Thus we assume that $\widetilde T\leq T_1$. Note that in \cite{Martens}, their choice of $\alpha$ was fixed while ours should depend on $\ep$.  
\item By Steps 1 and 2, we see that the desired conclusions of the theorem hold on the interval $[0,T']$. Thus in this step it is shown that in fact $T'$ can be taken to be uniformly bounded from below, only depending on $n,A,\tau,\mathfrak{v}_0$. This is the step in which EBL is used to say that there must exist some $r_0>0$ such that $B_{g(0)}(x,r_0)\subset B_{g(t)}(x,1)$ for all $(x,t)\in M\times [0,\widetilde T]$. This containment allows one to leverage the evolution of the curvature concentration against the known volume bounds at time $t=0$ to derive a contradiction if $\Lambda$ is large enough and $T'<\widetilde T$. 
\end{enumerate}

We now describe the necessary changes to the proofs of each of these three steps in our setting. For clarity of the proof, instead of defining the constants $\widetilde T, \Lambda, \sigma$ ahead of time, we will simply start with them all being equal to 1 and adjust them as we go only ever depending on $n,A,\tau,\mathfrak{v}_0$. Our first significant change is in the definition of $T'$. Rather than using the definition in \eqref{originalg(t)assumption}, we define $T'$ to be the maximal time in $[0,\min\{T, \widetilde T\}]$ such that 
\be\label{modifiedg(0)assumption}
\int_{B_{g(0)}(x, \frac{5}{4})}|\Rm|_{g(t)}^{\frac{n}{2}}\,dV_{g(t)}\leq \Lambda \ep^{\frac{n}{2}}
\ee
uniformly for all $(x,t)\in M\times [0,T']$. This subtle change of the domain being defined with respect to $g(0)$ rather than $g(t)$ allows for the application of our version of EBL (Lemma \ref{EBL}). Despite this change, we can still say that $T'>0$ because $g(t)$ is a bounded curvature flow (this argument follows by a slight modification of {\cite[Lemma 4.1]{Martens}}). Thus by restricting $\widetilde T$ and $\sigma$ sufficiently small, Step 1 of the proof in \cite{Martens} then follows synonymously. In particular, we obtain a uniform local $L^2$-Sobolev inequality along the flow: 
\be\label{step1sobolevinequality}
\left(\int_{B_{g(t)}(x,1)}u^{\frac{2n}{n-2}}\,dV_{g(t)}\right)^{\frac{n-2}{n}}\leq C_S(n,A,\tau)\int_{B_{g(t)}(x,1)}|\nabla^{g(t)}u|^2_{g(t)}+u^2\,dV_{g(t)},
\ee
for all $u\in W^{1,2}_0(B_{g(t)}(x,1))$ and all $(x,t)\in M\times [0,\min\{T', T_1\}]$. It is well known that an $L^2$-Sobolev inequality also implies a lower bound on volume ratios (for convenience, the reader can look to Lemma 2.1 of the same paper \cite{Martens}):
\be\label{lowerboundonvolumes}
\Vol_{g(t)}B_{g(t)}(x, r)\geq \tilde v(n, C_S) r^n \;\text{ for all } (x,t)\in M\times [0,\min\{T', T_1\}], \; r\in (0,1].
\ee
Recall that in \eqref{step1sobolevinequality} and \eqref{lowerboundonvolumes}, $T_1$ is the largest time in which the flow $g(t)$ enjoys an $\alpha/t$ curvature bound, with $\alpha=C_0\ep$ (here $C_0$ will be defined later in terms of $\Lambda$). Since $g(t)$ is a bounded curvature flow, we know that $T_1>0$, though a priori it appears to depend on the flow. In Step 2, however, we show that $T_1$ can be assumed to be at least as large as $T'$ whenever $\widetilde T$ and $\sigma$ are chosen small enough. It is necessary in this step to diverge from the limiting contradiction argument presented in \cite{Martens}. Instead, we utilize a more classical argument, in which we argue by contradiction by assuming that $T_1<T'$. By the maximality of $T_1$, the flow certainly does enjoy an $\alpha/t$ curvature bound on $[0,T_1]$, and thus condition \eqref{modifiedg(0)assumption} implies \eqref{originalg(t)assumption} on this interval by SBL, assuming $\widetilde T$ is small enough. At this point, we may directly apply {\cite[Proposition 2.1]{CCL}} to derive a contradiction if $\sigma$ is chosen small enough. Note that this argument (and constants thereof) is independent of the exact choice of $\ep\leq \sigma$.

\vspace{5pt}

It remains to show that Step 3 goes through as well. By taking $\widetilde T$ and $\sigma$ small enough, we may still arrive at the conclusion that 
\[
\int_{B_{g(t)}(x, 1/4)}|\Rm|_{g(t)}^{\frac{n}{2}}\,dV_{g(t)}\leq 4\cdot(4 \ep)^\frac{n}{2}
\]
holds uniformly for all $(x,t)\in M\times [0,T']$. Now by our modified assumption on the smallness of the local curvature concentration in \eqref{modifiedg(0)assumption}, our assumed initial volume upper bound \eqref{shorttimeboundedcurvvolume}, and our lower bound on volume ratios along the flow in \eqref{lowerboundonvolumes}, we can say by our version of EBL (Lemma \ref{EBL}) that if $\widetilde T$ is small enough, then there is some $r_0(n, A, \tau, \mathfrak{v}_0)>0$ such that for all $(x,t)\in M\times [0,T']$, there holds
\[
\int_{B_{g(0)}(x, r_0)}|\Rm|_{g(t)}^{\frac{n}{2}}\,dV_{g(t)}\leq 4\cdot (4 \ep)^\frac{n}{2}.
\]
Now the remainder of Step 3 follows verbatim to \cite{Martens} because by \eqref{shorttimeboundedcurvvolume} and \eqref{lowerboundonvolumes} we have a uniform doubling estimate on scales less than 4:
\[
\Vol_{g(0)}B_{g(0)}(x,2r)\leq 2^{3n+1} \mathfrak{v}_0 \tilde v\inv \Vol_{g(0)}B_{g(0)}(x,r)\leq  (2^{3n+1} \mathfrak{v}_0 \tilde v\inv ) \mathfrak{v}_0 r^n \;\text{ for all } r\in (0,4].
\]
Thus we obtain a contradiction if $\Lambda$ is large enough and if $T'$ is assumed to be strictly less than $\min\{T, \widetilde T\}$. The constant $C_0$ can therefore be taken to be $\Lambda^{\frac{2}{n}}$. For convenience in the proof of Theorem \ref{smoothingcorollary}, we will further assume that $\sigma$ is small enough so that $C_0\sigma\leq \frac{1}{4(n-1)}$. 
\end{proof}

\subsection{Proof of Theorem \ref{smoothingcorollary}}\hfill

\vspace{5pt}

For each $i\in \N$, let $g_i(t)$, $t\in [0,T]$ be the Ricci flow emanating from $g_i$ given by Theorem \ref{shortimeexistencemain} (note that $T$ here is independent of $i$). Without loss of generality, we will assume that $T\leq 1$, and for convenience of notation we denote $\alpha=C_0 \sigma$ which satisfies $\alpha<\frac{1}{4(n-1)}$ by the proof of Theorem \ref{shortimeexistencemain}. Thus the $g_i(t)$ satisfy all the following conditions, uniformly for all $(x,t)\in M_i\times [0,T]$:
\be\label{1.4propertiesofshorttime}
\begin{cases}
&\Vol_{g_i(0)}B_{g_i(0)}(x,r)\leq \mathfrak{v}_0 r^n \;\text{ for all } r\leq 1;\\
&\Vol_{g_i(t)}B_{g_i(t)}(x,r)\geq \tilde v(n,A,\tau) r^n\;\text{ for all } r\leq \frac{1}{16}; \\
&|\Rm|_{g_i(t)}(x)\leq \frac{\alpha}{t};\\
&\inj_{g_i(t)}(x)\geq C_0\inv \sqrt{t}; \text{ and }\\
&\left(\int_{B_{g_i(0)}(x,\frac{1}{16})} |\Rm|_{g_i(t)}^{\frac{n}{2}}\,dV_{g_i(t)}\right)^{\frac{2}{n}}\leq \alpha.
\end{cases}
\ee
By Hamilton's Compactness Theorem \cite{Hamilton}, some subsequence of the $(M_i,g_i(t))$, $t\in (0,T]$ converges in the Cheeger-Gromov sense to a smooth pointed complete Ricci flow $(M_\infty, g(t), x_\infty)$, $t\in (0,T]$, that also satisfies the properties in \eqref{shorttimeexistenceconclusions}. For the sake of brevity, we will continue to write $(M_i,g_i,p_i)$ for this subsequence, i.e., relabel so that $i_k=i$. The functions $F_i$ in the statement of the theorem are precisely those coming from the definition of the Cheeger-Gromov convergence. That is to say that $F_i: \Omega_i\to M_i$, which are diffeomorphisms onto their images, where $\Omega_i\subset \Omega_{i+1}\subset M_\infty$ is an exhausting sequence of open precompact sets with $F_i^* g_i(t) \to g(t)$ in $C^\infty_{loc}(M_\infty\times (0,T])$. Thus it suffices to prove the existence of the distance metric $d_0$ and measure $\mu$ as well as their convergence properties. We will deal with $d_0$ first.

\subsubsection{Existence and properties of $d_0$}\hfill

\vspace{5pt}

To show that this family of Ricci flows $g_i(t)$ are regular enough to extract a Gromov-Hausdorff limiting distance metric at time zero, we will use the following claim, which says that the conclusion of {\cite[Theorem 2.4]{HuangKongRongXu}} holds in this setting, i.e., that the flows $g_i(t)$ are uniformly (with respect to $i$) locally bi-H\"older.

\begin{claim}\label{GHcorclaim}
There exists some $C\geq 1$ depending only on $n,A,\tau,\mathfrak{v}_0$ such that for any $x_i,y_i\in M_i$ with $d_{g_i}(x_i,y_i)\leq 1$ and $t\in [0,T]$, there holds 
\be\label{GHcorobjective}
C\inv d_{g_i(t)}^{\frac{1}{1-2(n-1)\alpha}}(x_i, y_i)\leq d_{g_i}(x_i,y_i)\leq C d_{g_i(t)}^{\frac{1}{1+2(n-1)\alpha}}(x_i, y_i).
\ee
\end{claim}

\vspace{5pt}

By an identical argument to {\cite[Theorem 1.4]{CCL}} and {\cite[Theorem 5.1]{CHL}}, the claim implies the existence of a distance metric $d_0$ on $M_\infty$ such that some further subsequence of the $g_i$ (still calling them $g_i$) satisfies
\be\label{locallyuniformdistanceconv}
\lim_{i\to \infty}d_{F_i^* g_i}(x,y)= d_0(x, y)
\ee
for every pair of points $x,y\in M_\infty$. Thus we can say that the bi-H\"older property \eqref{GHcorobjective} decends to the limit as well:
\be\label{biholderlimit}
C\inv d_{g(t)}^{\frac{1}{1-2(n-1)\alpha}}(x, y)\leq d_{0}(x,y)\leq C d_{g(t)}^{\frac{1}{1+2(n-1)\alpha}}(x, y)
\ee
whenever $d_0(x,y)\leq 1$. This implies that $d_0$ generates the same topology as $M_\infty$. The fact that $(M_\infty,d_0)$ is a complete metric space follows from \eqref{biholderlimit} by an argument that we now sketch. If $(M_\infty, d_0)$ were not complete then we could choose some continuous rectifiable curve $\gamma:[0,1)\to M_\infty$, i.e., $\operatorname{Length}_{d_0}(\gamma)<\infty$, but with the property that $\lim_{s\uparrow 1}\gamma(s)\not\in M_\infty$. Now we may pick $\ep>0$ small enough so that $d_{g(t)}(x,y)\leq \ep$ implies that $d_0(x,y)\leq 1$ by the RHS of \eqref{biholderlimit}. Since $g(T)$ is complete, we may choose countably many points $\{x_j\}_{j\in \N}\subset \gamma([0,1))$ such that the $\{B_{g(T)}(x_j,\ep)\}_{j\in\N}$ are mutually disjoint, and we may write let $\{t_j\}_{j\in\N}\subset (0,1)$ be corresponding times with the property that $t_j$ is the latest time in which $\gamma(t_j)=x_j$. By the continuity of $\gamma$ with respect to $d_0$ and the fact $d_0$ generates the same topology as $g(T)$, we may select times $t_j'\in (t_j,1)$ which mark the first time after $t_j$ in which $\gamma$ intersects $\partial B_{g(T)}(x_j,\ep)$, i.e., 
\[
t_j'=\sup\{t>t_j : \gamma(s)\in B_{g(T)}(x_j,\ep) \;\text{ for all } s\in (t_j,t]\}.
\]
Then since $\ep$ was chosen small enough so that \eqref{biholderlimit} holds for $y\in B_{g(T)}(x_j,\ep)$, we obtain
\begin{align*}
\infty&=\sum_{j=1}^\infty C\inv \ep^{\frac{1}{1-2(n-1)\alpha}} =\sum_{j=1}^\infty C\inv [d_{g(T)}(x_j,\gamma(t_j'))]^{\frac{1}{1-2(n-1)\alpha}}\\&\leq\sum_{j=1}^\infty d_{0}(x_j,\gamma(t_j'))\leq \operatorname{Length}_{d_0}(\gamma)<\infty,
\end{align*}
a contradiction. Thus $(M_\infty, d_0)$ is a complete metric space. The claim about the pointed Gromov-Hausdorff convergence with the $F_i$ being the approximating functions is straightforward from \eqref{locallyuniformdistanceconv}. So it suffices to prove the claim.

\vspace{5pt}

\begin{proof}[Proof of Claim \ref{GHcorclaim}]
We will prove the inequalities in \eqref{GHcorobjective} individually. In what follows, we let $C$ denote any large constant depending only on $n,A, \tau, \mathfrak{v}_0,\alpha$ (recall that $\alpha=C_0\sigma$ also depends only on $n,A,\tau,\mathfrak{v}_0$) and may vary line-by-line, and we will also drop the subscripts $i$ for brevity. The RHS is identical to the proof of the analogous inequality in property (iv) of {\cite[Theorem 5.1]{CHL}}, so we only give a sketch. By the $\alpha/t$ curvature bound and SBL we obtain, for $0<s\leq t\leq T$,
\be\label{GHrhse1}
d_{g(0)}(x,y)\leq d_{g(s)}(x,y)+C\sqrt{s}\leq \left(\frac{t}{s}\right)^{(n-1)\alpha}d_{g(t)}(x,y)+C\sqrt{s}.
\ee
If $d_{g(t)}(x,y)>1$ then the RHS of \eqref{GHcorobjective} is obvious, so assume $d_{g(t)}(x,y)\leq 1$ and define $s\in (0,t]$ by
\be\label{GHrhse2}
s=t \,d_{g(t)}^{\frac{2}{1+2(n-1)\alpha}}(x,y).
\ee
Then the RHS of \eqref{GHcorobjective} follows by plugging \eqref{GHrhse2} directly into \eqref{GHrhse1} and simplifying using the assumptions that $t\leq 1$. \\

Now we prove the LHS. Recall from the proof of Theorem \ref{shortimeexistencemain} that we have some uniform lower bound of volume ratios up to some fixed scale along the flow (cf. \eqref{lowerboundonvolumes}). Thus by scaled versions of Lemma \ref{EBL}, we may obtain that 
\be\label{GHlhseebl}
d_{g(t)}(x,y)\leq C d_{g(0)}(x,y) \;\text{ whenever } \sqrt{\widehat T\inv t} \leq d_{g(0)}(x,y)\leq \frac{1}{32}
\ee
where $\widehat T$ is from Lemma \ref{EBL}, and depends on $n,A,\tau,\mathfrak{v}_0,\alpha$. If $t\geq\frac{ \widehat T}{32^2}$, then the $\alpha/t$ curvature bound allows us to conclude $d_{g(t)}(x,y)\leq C d_{g(0)}(x,y)$ immediately. If $\sqrt{\widehat T\inv t} \leq \frac{1}{32}\leq d_{g(0)}(x,y)\leq 1$, then we set $x_0=x$, $x_{32}=y$ and choose intermediate points $\{x_j\}_{j=1}^{31}$ satisfying
\[
\begin{cases}
&\sum_{j=0}^{31} d_{g(0)}(x_j,x_{j+1})\leq 32d_{g(0)}(x,y) ,\\
&\sqrt{\widehat T\inv t} \leq d_{g(0)}(x_j,x_{j+1})\leq \frac{1}{32}.
\end{cases}
\]
Then we may conclude $d_{g(t)}(x,y)\leq C d_{g(0)}(x,y)$ by applying \eqref{GHlhseebl} for each $j=0, 1, \dots, 31$, along with the triangle inequality. So it suffices to consider $x,y,t$ satisfying $d_{g(0)}(x,y)< \sqrt{\widehat T\inv t} < \frac{1}{32}$. In this case, write $\sqrt{s}=\sqrt{\widehat T} d_{g(0)}(x,y)$ so that by the $\alpha/t$ curvature bound and \eqref{GHlhseebl}, we obtain
\[
d_{g(t)}(x,y)\leq s^{-(n-1)\alpha}d_{g(s)}(x,y)\leq C s^{-(n-1)\alpha}d_{g(0)}(x,y)\leq Cd_{g(0)}^{1-2(n-1)\alpha}(x,y).
\]
This shows the LHS of \eqref{GHcorobjective} completes the proof of the claim
\end{proof}

\subsubsection{Existence and properties of $\mu$}\hfill

\vspace{5pt}

Our technique to establishing the existence of the measure $\mu$ will be to show that the limits $\lim_{t\to 0}\Vol_{g(t)}\Omega$ and $\lim_{i\to \infty}\Vol_{g_i} F_i(\Omega)$ both exist and are equal whenever $\Omega$ is a precompact Borel set. In the case that $\Omega$ is not precompact, we can then define $\mu(\Omega)$ as a limit of countably many precompact pieces. Before commencing the proof however, we give a reformulation of our Volume Perturbation Lemma \ref{volumeperturbation} in the context of Theorem \ref{shortimeexistencemain}.

\begin{lem}\label{VolPertwithThm}
Suppose $(M^n, g_0)$ satisfies the hypotheses of Theorem \ref{shortimeexistencemain} for some choices of variables $n\geq 4, A, \tau, \mathfrak{v}_0>0$, and write $g(t)$, $t\in [0,T]$ to be the corresponding complete Ricci flow. Given any precompact set $\Omega\subset \subset M$ such that $B_{g(0)}(x,r)\subset \Omega$ for some $r\in (0,\frac{1}{8}]$, then whenever $0\leq s\leq t\leq \min\{\frac{r^2}{\beta^2},T\}$, there holds 
\[
C^{-(t-s)}\Vol_{g(s)}\Omega\leq \Vol_{g(t)}\Omega\leq C^{(t-s)}\Vol_{g(s)}\Omega.
\]
Here $C>0$ is some constant that depends only on $n,A,\tau, r$ as well as the smallest number $N$ such that 
\[
\Omega\subset \bigcup_{j=1}^N B_{g(0)}\left(x_j, \frac{1}{16}\right).
\]
\end{lem}

\begin{proof}
The proof follows from the properties of such a flow $g(t)$ as laid out in \eqref{1.4propertiesofshorttime} and a direct application of Lemma \ref{volumeperturbation}. In particular, we can say as long as $t\in [0,T]$ satisfies $t\leq \frac{r^2}{\beta^2}$, or more generally $t\leq \frac{r^2}{4\beta^2\alpha}$ (recall $\alpha\leq \frac{1}{4(n-1)}$), we can apply SBL to see that
\[
B_{g(t)}(x,r/2)\subset B_{g(0)}(x,r)\subset \Omega.
\]
Since $r\leq \frac{1}{8}$, we have that $r_0:=r/2\leq \frac{1}{16}$ is a scale in which $g(t)$ has a lower bound of volume ratios $\tilde v$ from \eqref{1.4propertiesofshorttime}. Also, we can bound the integral scalar curvature on $\Omega$ as
\[
\left(\int_\Omega |\sR_{g(t)}|^{\frac{n}{2}}\,dV_{g(t)}\right)^{\frac{2}{n}}\leq \left(\sum_{j=1}^N\int_{B_{g(0)}(x_j,\frac{1}{16})}n^{\frac{n}{2}}|\Rm|_{g(t)}^{\frac{n}{2}}\,dV_{g(t)}\right)^{\frac{2}{n}}\leq (n^{\frac{n}{2}} N \alpha^{\frac{n}{2}})^{\frac{2}{n}}=N^{\frac{2}{n}} n \alpha. 
\]
Therefore, we can see that the result holds by applying Lemma \ref{volumeperturbation} with $r,\sigma$ there being $r_0, N^{\frac{2}{n}} n \alpha$, and thus we may take 
\[
C=\exp\left(N^{\frac{2}{n}} n \alpha  \tilde v^{-\frac{2}{n}}(r/2)^{-2}\right)\leq \exp\left(\frac{N^{\frac{2}{n}} n   \tilde v^{-\frac{2}{n}}}{(n-1)r^2}\right)
\]
where $\tilde v$ depends only on $n,A,\tau$. 
\end{proof}

\vspace{5pt}

We also require the following elementary lemma which will be applied with $f(t)$ being either $\Vol_{g_i(t)}F_i(\Omega)$ or $\Vol_{g(t)}\Omega$ in our proof. 

\begin{lem}\label{functionlemma}
Let $f: (0,T]\to \R$ be a positive continuous function satisfying
\[
C^{-(t-s) }f(s)\leq f(t) \leq C^{(t-s)} f(s)
\]
whenever $0<s\leq t\leq T$. Then for any $\ep>0$, there exists $t_\ep\in (0,T]$ depending only on $\ep$ as well as upper bounds of $C$ and $f(T)$, such that whenever $0<s\leq t\leq t_\ep$, there holds
\[
|f(t)-f(s)|\leq\ep.
\]
In particular, $\lim_{t\to 0}f(t)$ exists. 
\end{lem}

\begin{proof}
For any $\ep>0$, define
\[
t_\ep=\min\left\{T,\frac{1}{\log C}\log\left(1+\frac{\ep}{C^T f(T)}\right) \right\}.
\]
Then whenever $0<s\leq t\leq t_\ep$, we obtain
\[|f(t)-f(s)|\leq (C^{(t-s)}-1)\sup_{t'\in (0,T]}f(t')\leq(C^{t_\ep}-1) C^T f(T)\leq \ep.\]
\end{proof}

Armed with these preliminary lemmas, we are ready to prove the existence and equality of the desired limits. We split our proof is split into three claims: (1) First show that the time limit exists assuming that the interior of the set is not empty (Claim \ref{nonemtpyintclaim1}); Next show that the two limits are equal under the same assumption (Claim \ref{gilimitgtlimit}); Finally, remove the assumption about nonempty interior (Claim \ref{gilimitgtlimit2}). Note that it is unambiguous whether the `interior' of a set is with respect to the metric $d_0$ or the topology coming from the smooth structure of $M_\infty$ because we have already shown these are the same.

\begin{claim}\label{nonemtpyintclaim1}
For any precompact set $\Omega\subset M_\infty$ with nonempty interior, the limit
\[
\lim_{t\to 0}\Vol_{g(t)}\Omega
\]
exists.
\end{claim}

\begin{proof} 
Choose some $p\in M_\infty$ and $r\in (0,\frac{1}{8}]$ such that 
\[
B_{d_0}(p,2r)\subset \Omega.
\]
Then for $i$ large enough, we may insist that
\[
B_{g_i(0)}(F_i(p),r)\subset F_i\left(B_{d_0}(p,2r)\right)\subset F_i(\Omega).
\]
For this fixed choice of $r>0$, we will denote $T_\Omega=\min\{\frac{ r^2}{\beta^2}, T\}$. Now since $\Omega\subset M_\infty$ is precompact, we may take finitely many points $\{x_j\}_{j=1}^N\subset M_\infty$ such that 
\[
\Omega\subset\bigcup_{j=1}^N B_{g(T_\Omega)}\left(x_j, \frac{1}{100}\right).
\]
Then restrict $i\in\N$ large enough so that for each $1\leq j\leq N$,
\[
F_i\left(B_{g(T_\Omega)}\left(x_j, \frac{1}{100}\right)\right)\subset B_{g_i(T_\Omega)}\left(F_i(x_j), \frac{1}{50}\right).
\]
Then by our definition of $T_\Omega$, our assumption of $ r\leq \frac{1}{8}$, and SBL, we can say that for all $i$ sufficiently large, 
\[
F_i(\Omega)\subset\bigcup_{j=1}^N B_{g_i}\left(F_i(x_j), \frac{1}{16}\right).
\]
Then by Lemma \ref{VolPertwithThm}, we can say that there exists a constant $C>0$ depending on $n,A,\tau, N$ (but not on $i$) such that whenever $0\leq s\leq t\leq T_\Omega$, 
\be\label{gibothsidesbound}
C^{-(t-s)}\Vol_{g_i(s)}F_i(\Omega)\leq \Vol_{g_i(t)}F_i(\Omega)\leq C^{(t-s)}\Vol_{g_i(s)}F_i(\Omega).
\ee
So for any fixed $0<s\leq t\leq T_\Omega$, we can take $i\to \infty$, which shows that \eqref{gibothsidesbound} transcends to the limit:
\[
C^{-(t-s)}\Vol_{g(s)}\Omega\leq \Vol_{g(t)}\Omega\leq C^{(t-s)}\Vol_{g(s)}\Omega \;\text{ whenever } 0<s\leq t\leq T_\Omega. 
\]
We can therefore conclude that $\lim_{t\to 0}\Vol_{g(t)} \Omega$ exists by a direct application of Lemma \ref{functionlemma}.
\end{proof}

\vspace{5pt}

\begin{claim}\label{gilimitgtlimit}
For any precompact set $\Omega\subset M_\infty$ with nonempty interior, there holds
\[
\lim_{i\to \infty}\Vol_{g_i}F_i(\Omega)=\lim_{t\to 0}\Vol_{g(t)}\Omega.
\]
\end{claim}
\begin{proof} 
We have already established that the limit $\lim_{t\to 0}\Vol_{g(t)}\Omega$ exists. Thus for any $\ep>0$, we may assume $t_0>0$ is small enough so that whenever $s\in (0,t_0]$,
\[
|\lim_{t\to 0}\Vol_{g(t)}\Omega-\Vol_{g(s)}\Omega|<\ep.
\]
Without loss of generality, we assume $t_0\leq T_\Omega$ where $T_\Omega$ is defined in the proof of Claim \ref{nonemtpyintclaim1}. In what follows, we assume that $i\geq i_\Omega$ for some $i_\Omega\in \N$ which may be increased several times. First, if we restrict $i_\Omega$ large enough so that 
\[
\Vol_{g_i(t_0)}F_i(\Omega)\leq 2 \Vol_{g(t_0)}\Omega,
\]
then we obtain from \eqref{gibothsidesbound} and Lemma \ref{functionlemma} some small $t_\ep\in (0,t_0]$ (here $t_\ep$ does not depend on $i$ as long as $i\geq i_\Omega$) satisfying
\[
|\Vol_{g_i(t_\ep)}F_i(\Omega)-\Vol_{g_i(0)}F_i(\Omega)|<\ep.
\]
Then we increase $i_\Omega$ again if necessary to ensure that 
\[
|\Vol_{g(t_\ep)}\Omega-\Vol_{g_i(t_\ep)}F_i(\Omega)|<\ep.
\]
A direct application of the triangle inequality completes the proof.
\end{proof}

\vspace{5pt}

\begin{claim}\label{gilimitgtlimit2}
For any precompact Borel set $\Omega\subset M_\infty$ (no other assumptions), there holds
\[
\lim_{i\to \infty}\Vol_{g_i}F_i(\Omega)=\lim_{t\to 0}\Vol_{g(t)}\Omega.
\]
\end{claim}
\begin{proof} 
We begin by claiming that
\be\label{Borele1}
\liminf_{t\to 0}\Vol_{g(t)} \Omega=\limsup_{r\to 0}\lim_{t\to 0}\Vol_{g(t)}(\Omega\cup B_{d_0}(x_\infty, r)).
\ee
The ``$\leq$" inequality is obvious, so assume for the sake of contradiction that there is some $\ep>0$ such that 
\be\label{Borele1.1}
\ep+\liminf_{t\to 0}\Vol_{g(t)} \Omega\leq \limsup_{r\to 0}\lim_{t\to 0}\Vol_{g(t)}(\Omega\cup B_{d_0}(x_\infty, r)).
\ee
Define $r_0=\sqrt[n]{\frac{\ep}{2^{n+4}\mathfrak{v}_0}}$, i.e., $2^{n+4}\mathfrak{v}_0r_0^n=\ep$. Then select some $r_1\in (0,r_0)$ small enough so that 
\be\label{Borele2}
\limsup_{r\to 0}\lim_{t\to 0}\Vol_{g(t)}(\Omega\cup B_{d_0}(x_\infty, r))\leq \lim_{t\to 0}\Vol_{g(t)}(\Omega\cup B_{d_0}(x_\infty, r_1))+\frac{\ep}{4}.
\ee
Since the limit on the RHS of \eqref{Borele2} exists by Claim \ref{gilimitgtlimit}, we may take $t_0>0$ small enough so that whenever $s\in (0,t_0]$,
\be\label{Borele3}
\lim_{t\to 0}\Vol_{g(t)}(\Omega\cup B_{d_0}(x_\infty, r_1))\leq \Vol_{g(s)}(\Omega\cup B_{d_0}(x_\infty, r_1))+\frac{\ep}{4}.
\ee
Similarly, we can find a sequence $t_j\to 0$ such that for every $j\in\N$, 
\be\label{Borele4}
\Vol_{g(t_j)}\Omega-\frac{\ep}{4}\leq\liminf_{t\to 0}\Vol_{g(t)} \Omega.
\ee
Thus whenever $t_j\leq t_0$, we can see by \eqref{Borele1.1} - \eqref{Borele4} that
\be\label{Borele5}
\frac{\ep}{4} +\Vol_{g(t_j)}\Omega\leq \Vol_{g(t_j)}(\Omega\cup B_{d_0}(x_\infty, r_1))\leq
\Vol_{g(t_j)}\Omega+\Vol_{g(t_j)} (B_{d_0}(x_\infty, r_1)).
\ee
For any $j\in\N$, we can assume that whenever $i$ is large enough (depending on $j$), we have
\[
\Vol_{g(t_j)} (B_{d_0}(x_\infty, r_1))\leq 2\Vol_{g_i(t_j)}F_i(B_{d_0}(x_\infty, r_1))\leq 2\Vol_{g_i(t_j)}(B_{g_i}(x_i, 2r_1)).
\]
Restricting $j$ large enough so that $t_j\leq \frac{(2r_1)^2}{\beta^2}$, we obtain from Lemma \ref{VolPertwithThm} and \eqref{Borele5} that
\[
\frac{\ep}{4}\leq 2\Vol_{g_i(t_j)}(B_{g_i}(x_i, 2r_1))\leq 4\Vol_{g_i}(B_{g_i}(x_i, 2r_1)).
\]
Here we have assumed that $i$ is sufficiently large (depending on $j$), and that $j$ in turn is large enough to ensure that $C^{t_j}\leq 2$, where $C$ is the constant from Lemma \ref{VolPertwithThm}. Now a contradiction is immediately apparent because we have assumed $r_1<r_0$, and we know that $\Vol_{g_i}(B_{g_i}(x_i, 2r_1))\leq \mathfrak{v}_0(2r_1)^n$. Thus we have proved that \eqref{Borele1} holds. Moreover, using a similar (albeit less involved) approach to the proof of \eqref{Borele1}, we also obtain
\be\label{Borele6}
\liminf_{i\to \infty}\Vol_{g_i} F_i(\Omega)=\limsup_{r\to 0}\lim_{i\to \infty}\Vol_{g_i}F_i(\Omega\cup B_{d_0}(x_\infty, r)).
\ee
Not only can we say from \eqref{Borele1} and \eqref{Borele6} that the limits $\lim_{i\to \infty}\Vol_{g_i} F_i(\Omega)$ and $\lim_{t\to 0}\Vol_{g(t)}$ exist, but we also can say that for any $\ep>0$, we can take $r>0$ small enough so that 
\ba
|\lim_{i\to \infty}\Vol_{g_i}F_i(\Omega)-&\lim_{t\to 0}\Vol_{g(t)}\Omega|\leq |\lim_{i\to \infty}\Vol_{g_i}F_i(\Omega)-\lim_{i\to \infty}\Vol_{g_i}F_i(\Omega\cup B_{d_0}(x_\infty, r))|\\&\quad+|\lim_{i\to \infty}\Vol_{g_i}F_i(\Omega\cup B_{d_0}(x_\infty, r))-\lim_{t\to 0}\Vol_{g(t)}(\Omega\cup B_{d_0}(x_\infty, r))|\\&\quad+|\lim_{t\to 0}\Vol_{g(t)}(\Omega\cup B_{d_0}(x_\infty, r))-\lim_{t\to 0}\Vol_{g(t)}\Omega|\leq \ep+0+\ep.
\end{align*}
This completes the proof of the claim. 
\end{proof}

\vspace{5pt}

We are now ready to formally define the measure $\mu$. For the sake of brevity, let us write $\sA_j$ to denote the $j$'th annulus centered at $x_\infty$ with respect to the metric $d_0$, i.e., 
\[
\sA_j:=B_{d_0}(x_\infty, j)\setminus B_{d_0}(x_\infty, j-1).
\] 
Also we write $\operatorname{Borel}(\Omega)$ to denote the $\sigma$-algebra of all Borel subset of $\Omega$ and recall a basic measure theory fact that 
\[
\operatorname{Borel}(\sA_j)=\{E\cap \sA_j : E\subset\operatorname{Borel}(M_\infty)\}.
\]
Thus for each $j\in\N$, we define $\mu_j: \operatorname{Borel}(\sA_j)\to \R$ by 
\be\label{defnofmu}
\mu_j(E)=\lim_{t\to 0}\Vol_{g(t)}E \;\; \text{ (or equivalently, } \mu_j(E)=\lim_{i\to \infty}\Vol_{g_i}F_i(E) \text{)},
\ee
which is a well-defined function by Claim \ref{gilimitgtlimit2}. Additionally, by an elementary measure theory argument (see \cite{MSE} for instance), we may conclude that each $\mu_j$ is countably additive, or equivalently in this setting, that each tuple $(\sA_j, \operatorname{Borel}(\sA_j), \mu_j)$ is in fact a measure space. Now for any $E\in \operatorname{Borel}(M_\infty)$, we define
\[
\mu(E):=\sum_{j=1}^\infty \mu_j(E\cap \sA_j).
\]
This can be seen to define a measure on $M_\infty$ because if $E=\cup_{k=1}^\infty E_k$ for some pairwise disjoint sequence of Borel sets $E_k$, we have
\[
\mu(E)=\sum_{j=1}^\infty \mu_j(E\cap \sA_j)=\sum_{j=1}^\infty \sum_{k=1}^\infty\mu_j(E_k\cap \sA_j)=\sum_{k=1}^\infty \sum_{j=1}^\infty\mu_j(E_k\cap \sA_j)=\sum_{k=1}^\infty\mu(E_k),
\]
where we have applied Tonelli's Theorem and the countable additivity of each $\mu_j$. Finally, to show that $\mu$ is Radon, we must show that it is outer and inner regular, i.e., that
\[
\mu(E)=\sup\{\mu(K) : K\subset E \text{ is compact}\}=\inf\{\mu(U) : U\supset E \text{ is open}\}
\]
for any $E\in \operatorname{Borel}(M_\infty)$. Since the proof of inner regularity is very similar to that of the outer regularity, we will only show the latter. To that end, fix some $E\in \operatorname{Borel}(M_\infty)$ and $\ep>0$. Just as in the proof of \eqref{Borele1} in Claim \ref{gilimitgtlimit2}, for each $j\in\N$ we can find some $x_j\in M_\infty$ and $r_j>0$ such that $B_{d_0}(x_j, r_j)\subset \sA_j$ and that
\be\label{Radon0}
\mu(B_{d_0}(x_j, r_j))\leq \frac{\ep}{2^{j+2}}.
\ee
Write $B_j=(E\cap \sA_j) \cup B_{d_0}(x_j, r_j)$, which is precompact in $M_\infty$, so we may choose a finite sequence of points $\{y_k\}_{k=1}^N\subset M_\infty$ such that 
\be\label{containmentRadon}
B_j\subset \bigcup_{k=1}^N B_{d_0}\left(y_k, \frac{1}{100}\right).
\ee
Now choose $t_j'>0$ small enough such that whenever $t\in (0,t_j']$, 
\be\label{Radon1}
|\mu_j(B_j)-\Vol_{g(t)}(B_j)|<\frac{\ep}{2^{j+2}}.
\ee
Using the fact that $\Vol_{g(t)}$ is a Riemannian measure (and subsequently, Radon), we may choose some precompact open set $U_{j,t}\subset M_\infty$, containing $B_j$, that satisfies
\be\label{Radon2}
\Vol_{g(t)}(U_{j,t})-\Vol_{g(t)}(B_j)<\frac{\ep}{2^{j+2}}.
\ee
Moreover by intersecting $U_{j,t}$ with the finite open cover from \eqref{containmentRadon}, we may assume
\[
U_{j,t}\subset \bigcup_{k=1}^N B_{d_0}\left(y_k, \frac{1}{100}\right).
\]
Importantly, here $N$ is independent of $t\in (0,t_j']$. Thus by a similar argument to that of the proof of Claim \ref{gilimitgtlimit2}, we obtain that \be\label{Radon3}
|\Vol_{g(t)}(U_{j,t})-\mu(U_{j,t})|<\frac{\ep}{2^{j+2}}
\ee
whenever $t$ is sufficiently small. So for each $j\in \N$, we select $t_j>0$ small enough so that all of \eqref{Radon1}, \eqref{Radon2}, \eqref{Radon3} hold. Then with $U=\bigcup_{j=1}^\infty U_{j,t_j}$ (which clearly is open and contains $E$ because each $U_{j,t_j}$ contains $E\cap \sA_j$), we may apply \eqref{Radon0} and \eqref{Radon1} - \eqref{Radon3} along with the fact that $\mu|_{\sA_j}=\mu_j$ to see that
\ba
\mu(U)-\mu(E)&\leq \sum_{j=1}^\infty \mu(U_{j,t_j})-\sum_{j=1}^\infty\mu_j(B_j)+\sum_{j=1}^\infty\mu_j(B_{d_0}(x_j,r_j))\\&\leq
 \sum_{j=1}^\infty \left[\Vol_{g(t_j)}(U_{j,t_j})+\frac{\ep}{2^{j+2}}\right]-\sum_{j=1}^\infty\left[\Vol_{g(t_j)}(B_j)-\frac{\ep}{2^{j+2}}\right]+\sum_{j=1}^\infty \frac{\ep}{2^{j+2}}\\&\leq  \sum_{j=1}^\infty \left[\frac{\ep}{2^{j+2}}+\frac{\ep}{2^{j+2}}+\frac{\ep}{2^{j+2}}+\frac{\ep}{2^{j+2}}\right]=\ep.
\end{align*}
This completes the proof. 

\vspace{5pt}

\subsection{Proof of Theorem \ref{longtimecorollary}}\hfill

\vspace{5pt}

We begin by assuming that $\sigma(n,A,\mathfrak{v}_0)$ is bounded from above by $\sigma(n,A,1,\mathfrak{v}_0)$ from the statement of Theorem \ref{shortimeexistencemain} (in other words, set $\tau=1$), and we will further shrink $\sigma$ if necessary in the proofs of properties \eqref{longtimecorconc1} and \eqref{longtimecorconc2}. Because the assumptions of the theorem are scale-invariant, we may apply Theorem \ref{shortimeexistencemain} to successively extreme blowdowns $R^{-2}g$, which (after scaling back) gives a family of Ricci flows $g_R(t)$ on $M\times [0, TR^2]$, for $R$ arbitrarily large. Note in particular that we have used properties of the local $\bar\nu$-functional:
\[
\bar\nu(B_{R^{-2}g}(x,1), R^{-2}g, 1)=\bar\nu(B_{g}(x,R), g, R^2)\geq \bar\nu(M, g)\geq -A.
\]
Then just as in the proof of Theorem \ref{shortimeexistencemain}, we can extract a converging subsequence
\[
g_{R_j}(t) \xrightarrow{C^{\infty}_{loc}(M\times [0,\infty))} g(t),
\]
where $g(0)=g$. By virtue of the conclusions \eqref{shorttimeexistenceconclusions} being rescaled as well, we can conclude that the long-time complete flow $g(t)$ satisfies the following for all $t\in (0,\infty)$:
\be\label{longtimecorconclusions}
\begin{cases}
&|\Rm|_{g(t)}\leq C_0 \sigma t\inv; \\
&\inj_{g(t)}\geq C_0\inv \sqrt t;\text{ and}\\
&\left(\int_{M} |\Rm|_{g(t)}^{\frac{n}{2}}\,dV_{g(t)}\right)^{\frac{2}{n}}\leq C_0\sigma.
\end{cases}
\ee
Here $C_0\geq 1$ depends only on $n,A,\mathfrak{v}_0$ (recall we fixed $\tau=1$ here). Moreover, by the proof of Theorem \ref{shortimeexistencemain}, we can also conclude that we have a global $L^2$-Sobolev inequality and a uniform lower bound on all volume ratios along the flow. More precisely, we know that \eqref{step1sobolevinequality} holds for the metrics $R^{-2}g$. So as $R\to \infty$, we obtain
\be\label{Sobinequalityforalltime}
\left(\int_M |u|^{\frac{2n}{n-2}}\,dV_{g(t)}\right)^{\frac{n-2}{n}}\leq C_S\int_M |\nabla^{g(t)} u|^2_{g(t)}\,dV_{g(t)}
\ee
for all compactly supported $u\in W^{1,2}(M,g(t))$, where $C_S(n,A,1)$ is independent of $t\in [0,\infty)$. Similarly, by rescalings of \eqref{lowerboundonvolumes}, we obtain
\be\label{uniformlypositiveVR}
\inf_{x\in M, r>0, t\geq 0}\frac{\Vol_{g(t)}B_{g(t)}(x,r)}{r^n}\geq \tilde v(n,C_S) .
\ee
Now with these properties we are ready to show that conclusions \eqref{longtimecorconc1}, \eqref{longtimecorconc2}, \eqref{longtimecorconc3} hold. \\

\begin{prop}
Conclusion \eqref{longtimecorconc1} holds.
\end{prop}

\begin{proof}
The proof is similar to the proof of {\cite[Theorem 1.2]{CHL}}. First note that since assumption \eqref{longtime1} as well as each of the conclusions in \eqref{longtimecorconclusions} and \eqref{uniformlypositiveVR} are scale-invariant, the flows $g_R(t)=R^{-2}g(R^2t)$ satisfy these conclusions as well. Lemma \ref{EBL} then implies that whenever $t\in [0, \widehat{T}(n, C_0 \sigma, \tilde v, C_0\sigma)]$, we have
\[
B_{g_R(0)}(x_0, 1)  \subset B_{g_R(t)}(x_0, 2^{2n+2} \mathfrak{v}_0\tilde v\inv).
\]
Fix some $T> \widehat T$ large enough so that $C_0\inv \sqrt T \geq 2^{2n+3} \mathfrak{v}_0\tilde v\inv$, then shrink $\sigma$ accordingly so that for any Ricci flow $h(t)$, $t\in [0,T]$ satisfying \eqref{longtimecorconclusions} (in particular, the $g_R(t)$), 
\[
d_{h(T)}(x_0, \cdot)\leq 2d_{h(\widehat T)}(x_0, \cdot).
\]
Scaling back, we can say that $B_{g(0)}(x_0, R)$ is homeomorphic to $\R^n$ for all $R>0$. The fact that $M$ is homeomorphic to $\R^n$ follows by \cite{Brown}, and the homeomorphism can be taken to be a diffeomorphism when $n\neq 4$ because $\R^4$ is only Euclidean space with exotic differential structures. 
\end{proof}

\vspace{5pt}

\begin{prop}
Conclusion \eqref{longtimecorconc2} holds.
\end{prop}

\begin{proof}
The proof is similar to the proof of {\cite[Theorem 1.2]{Martens}}. First, we further shrink $\sigma$ if necessary to ensure that 
\[
\sigma\leq\min\left\{ \frac{\delta}{C_0 C_S}, \frac{1}{4\beta^2C_0}\right\}
\]
where $\delta=\delta(n)$ is the dimensional constant from {\cite[Theorem 1.2]{CM}}. We may now consider $\sigma$ to be fixed. Then by the smallness of the global curvature concentration in \eqref{longtimecorconclusions} and the Sobolev inequality \eqref{Sobinequalityforalltime}, we may apply {\cite[Theorem 1.2]{CM}} to the bounded curvature metric $g(1)$ to conclude faster than $1/t$ curvature decay:
\be\label{fasterthan1/tdecay}
\lim_{t\to \infty}t\sup_{x\in M}|\Rm|_{g(t)}(x)=0.
\ee
Now we argue by contradiction by assuming there is some $\ep>0$ and sequences $x_j\in M$ and $R_j\to \infty$ with 
\be\label{longtimebassumption}
\Vol_g B_g(x_j, R_j)< (1-\ep) \omega_n R_j^n.
\ee
Consider the blowdown pointed flows $(M, g_j(t), x_j)$, $t\in (0,1]$, where $g_j(t)=R_j^{-2} g(R_j^2 t)$. By the estimates \eqref{longtimecorconclusions}, we obtain - via Hamilton's Compactness \cite{Hamilton} - that $g_j(t)$ subconverges to a limiting complete flow $(M_\infty, g_\infty(t))$, $t\in (0,1]$. Furthermore, by nature of the uniform faster than $1/t$ curvature decay \eqref{fasterthan1/tdecay} and uniform lower bound on volume ratios \eqref{uniformlypositiveVR}, we can see that this flow must be isometric to the static flow on flat Euclidean space. Take some $t_\ep$ small enough (depending on $\ep, C_0, \sigma, \tilde v$ and thus only on $\ep, n, A, \mathfrak{v}_0$) such that 
\be\label{conditionontep}
\min\left\{\exp\left(-4n C_0 \sigma \tilde v^{-\frac{2}{n}}t_\ep\right),\left(1-\beta\sqrt{C_0\sigma t_\ep}\right)^n\right\}\geq (1-\ep)^{\frac{1}{3}}.
\ee
Then we can say that for $j$ sufficiently large, we have
\ba
\Vol_{g_j(0)}B_{g_j(0)}(x_j, 1)&\geq (1-\ep)^{\frac{1}{3}}\Vol_{g_j(t_\ep)}B_{g_j(0)}(x_j, 1)\\&\geq (1-\ep)^{\frac{1}{3}}\Vol_{g_j(t_\ep)}B_{g_j(t_\ep)}\left(x_j, 1-\beta\sqrt{C_0\sigma t_\ep}\right)\\&\geq
(1-\ep)^{\frac{2}{3}}\left(1-\beta\sqrt{C_0\sigma t_\ep}\right)^n \omega_n\\&\geq
(1-\ep)\omega_n.
\end{align*}
Here we have used Lemma \ref{volumeperturbation} along with $|\sR_g|\leq n|\Rm|_g$ and the first condition on $t_\ep$ in \eqref{conditionontep} in the first line, SBL in the second line, the pointed convergence $(M,g_j(t_\ep),x_j)\to (\R^n, g_E, 0)$ in the third line, and the second condition on $t_\ep$ in \eqref{conditionontep} in the final line. After scaling back, we can see immediately that this contradicts our assumption \eqref{longtimebassumption}, whenever $R_j$ is sufficiently large. 
\end{proof}

\begin{remark}
Precisely how large $R$ needs to be to derive a contradiction in the preceding proof not only depends on the variables $\ep, n, A, \mathfrak{v}_0$, but also on the rate of the convergence in \eqref{fasterthan1/tdecay}, which in turn possibly depends on some other properties of $g$ not captured by the assumptions of the theorem. 
\end{remark}

\begin{prop}
Conclusion \eqref{longtimecorconc3} holds.
\end{prop}

\begin{proof}
We continue using the notation of $\sM_\ep$ as in the theorem statement where $n,A,\mathfrak{v}_0$ are all fixed. We first have the following claim:

\vspace{5pt}

\begin{claim}\label{partcintermediateclaim}
For any $\delta>0$, we can take $\ep$ small enough (depending only on $\delta, n, A, \mathfrak{v}_0$), so that for any $(M,g)\in \sM_\ep$ and associated long-time Ricci flow $g(t)$ as given above, there holds
\be\label{longtimeclaimc1}
\Vol_{g(t)}B_{g(t)}(x, \sqrt{t})\geq (1-\delta) \omega_n t^{\frac{n}{2}}
\ee
for all $(x,t)\in M\times (0,\infty)$.
\end{claim}

\begin{proof}[Proof of Claim]
Since the hypotheses and conclusion of the claim are invariant under scaling, it suffices to prove \eqref{longtimeclaimc1} holds only for some short time interval $t\in [0,T_1]$. Then since we are working locally, we may assume that the flow we are working on is actually complete with \emph{bounded curvature} and satisfying the estimates \eqref{shorttimeexistenceconclusions}. This assumption is validated by the fact that our (possibly initially unbounded curvature) Ricci flow is obtained by a smooth local limit of such flows (recall the proof of Theorem \ref{shortimeexistencemain}). The proof is then identical to the beginning of the proof of {\cite[Theorem 4.2]{CCL}}. 
\end{proof}

The remainder of the proof is similar to the proof of part \eqref{longtimecorconc2} with the notable difference being that instead of choosing a time $t_\ep$ sufficiently small depending on $\delta$, we take $\ep$ small instead. Thus for some given $\delta>0$, let $\ep_{\delta/2}$ be as given in Claim \ref{partcintermediateclaim}. Again by the invariance under scaling, we may assume that $r=1$. First assume that $\ep$ is small enough so that 
\[
\ep\leq \min\left\{\ep_{\delta/2},\frac{1}{4\beta^2 C_0}\right\}.
\]
Then writing $t_\ep=\left(1-\beta\sqrt{C_0\ep}\right)^2<1$, we apply Lemma \ref{volumeperturbation}, together with SBL and Claim \ref{partcintermediateclaim} to see that 
\ba
(1-\delta/2)\omega_n t_\ep^{\frac{n}{2}} &\leq \Vol_{g(t_\ep)}B_{g(t_\ep)}(x, \sqrt{t_\ep})\\&\leq \Vol_{g(t_\ep)}B_{g(0)}(x, 1)\\&\leq \exp\left(4nC_0\ep \tilde v^{-\frac{2}{n}}t_\ep\right)\Vol_{g(0)}B_{g(0)}(x, 1).
\end{align*}
Then the result follows if we further shrink $\ep$ if necessary to ensure that 
\[
(1-\delta/2) \left(1-\beta \sqrt{C_0 \ep}\right)^n \exp\left(-4nC_0\ep \tilde v^{-\frac{2}{n}}\right)\geq 1-\delta.
\]
\end{proof}

\end{document}